\documentclass[11pt]{amsart}

\usepackage{amsmath,amssymb,amsthm}
\usepackage{amsmath,amssymb,amsthm,amscd}
\usepackage[mathscr]{eucal}
\numberwithin{equation}{section}
\newtheorem{prop}{Proposition}
\newtheorem{theo}[prop]{Theorem}
\newtheorem{lemm}[prop]{Lemma}
\newtheorem{coro}[prop]{Corollary}

\newtheorem*{LST}{Theorem (Local Surjectivity Theorem)}
\newtheorem*{LMthm}{Theorem (Lax--Milgram Theorem)}
\newtheorem*{FAthm}{Theorem (Fredholm alternative)}

\theoremstyle{definition}

\newtheorem*{pflemism}{Proof of Lemma~\ref{le:ism}}

\newtheorem*{pfletori}{Proof of Lemma~\ref{le:tori}}
\newtheorem*{pfletoriext}{Proof of Lemma~\ref{le:toriext}}

\newtheorem{defi}[prop]{Definition}
\newtheorem*{ack}{Acknowledgment}
\theoremstyle{remark}

\newcommand{\p}{\partial}
\newcommand{\pp}{\sqrt{-1}\partial \bar{\partial}}
\newcommand{\ppr}{(\sqrt{-1}/2) \partial \bar{\partial}}
\newcommand{\ppfr}{\frac{\sqrt{-1}}{2} \partial \bar{\partial}}
\newcommand{\im}{\textup{Im}}

\def\lab{\label}

\begin{document}

\title[Form--type Calabi--Yau equations]{Form--type Calabi--Yau equations}

\author{Jixiang Fu}
\address{Institute of Mathematics\\ Fudan University \\ Shanghai
200433, China} \email{majxfu@fudan.edu.cn}
\author{Zhizhang Wang}
\address{Institute of Mathematics\\ Fudan University \\ Shanghai
200433, China} \email{youxiang163wang@163.com}
\author{Damin Wu}
\address{Department of Mathematics \\
         The Ohio State University \\
         1179 University Drive, Newark, OH 43055, U.S.A.}
\email{dwu@math.ohio-state.edu}

\begin{abstract}
Motivated from mathematical aspects of the superstring theory,
we introduce a new equation on a balanced, hermitian manifold,
with zero first Chern class. Solving the equation, one will obtain,
in each Bott--Chern cohomology class, a balanced metric which is hermitian Ricci--flat. T
his can be viewed as a differential form level generalization of the classical Calabi--Yau equation.
We establish the existence and uniqueness of the equation on complex tori, and prove certain uniqueness and
openness on a general K\"ahler manifold.
\end{abstract}
\maketitle

\section{Setting and Equations} \label{se:Setting}
In the superstring theory, the internal space $X^3$ is a complex
three-dimensional manifold with a non-vanishing holomorphic
three-form $\Omega$ \cite{Str} (cf. \cite{BBFTY}). The $N=1$ supersymmetry requires
\cite{Str,LY1}
\begin{equation*}
d(\parallel\Omega\parallel_\omega\omega^2)=0,
\end{equation*}
for some hermitian metric (form) $\omega$. The above equation in
mathematics says that $\omega$ is a conformally balanced metric. (We
recall  that \cite{Mic} a hermitian metric $\omega$ on an
$n$-dimensional complex manifold $X^n$ is called \emph{balanced} if
$\omega$ satisfies that
\[
   d (\omega^{n-1}) = 0 \qquad \textup{on $X^n$.\,)}
\]
Note that \cite{GP, FY} the torus bundles
over $K3$ surfaces and over complex abelian surfaces twisted by two
anti-self dual $(1,1)$-forms admit a non-vanishing holomorphic
three-form $\Omega$ and a natural balanced metric $\omega_0$ such
that
\begin{equation} \label{eq:CY}
   \parallel\Omega\parallel_{\omega_0}=1.
\end{equation}
As important examples in the superstring theory and
non-K\"ahler complex geometry, the complex manifolds
$\#_k(S^3\times S^3)$ for any $k\geq 2$ \cite{Fri2,LT} also admit a
non-vanishing holomorphic three-form \cite{Fri2} and a balanced
metric \cite{FLY}. Moreover, we know that $\#_k(S^3\times S^3)$
satisfies the $\partial\bar\partial$--lemma \cite{Fri2}. A natural
question to ask is, whether $\#_k(S^3 \times S^3)$ admits a balanced
metric $\omega_0$ such that \eqref{eq:CY} holds. Such a metric
$\omega_0$, if exists, will play an important role in the
superstring theory and hermitian geometry.

More generally, let $X^n$ $(n\geq 3)$ be a complex $n$-dimensional
manifold with a non-vanishing holomorphic $n$-form $\Omega$ and with
a balanced metric $\omega_0$.
 We want to
look for a balanced metric $\omega$ such that
\begin{equation} \lab{eq:omgn}
\omega^{n-1}=\omega_0^{n-1}+\frac{\sqrt{-1}}{2}\partial\bar\partial
\varphi,
\end{equation}
for some real $(n-2,n-2)$--form $\varphi$, and such that
\begin{equation} \label{eq:FCY}
   \textup{$\|\Omega\|_{\omega} = $ some positive constant $C_0$}.
\end{equation}
In other words, we would like to find solutions of \eqref{eq:FCY} in the cohomology class $[\omega_0^{n-1}]\in H^{n-1,n-1}_{BC}(X)$. Here $H^{p,q}_{BC}(X)$ stands for the Bott--Chern cohomology:
$$H^{p,q}_{BC}(X)=\frac{(\ker\partial\cap\ker\bar\partial)\cap\Omega^{p,q}(X)}{\textup{im}\ \partial\bar\partial\cap\Omega^{p,q}(X)}.$$
One can certainly normalize the constant $C_0$ in \eqref{eq:FCY} to be 1, as in \eqref{eq:CY}. However, it may be more convenient to set
\begin{equation*} \lab{eq:formCY}
    C_0 = \left(\int_X \omega^n \right)^{-\frac{1}{2}},
\end{equation*}
from the equation point of view.
As in the K\"ahler case, equation \eqref{eq:FCY} is equivalent to the
equation
\begin{equation} \lab{104}
\frac{\det
\omega}{\det\omega_0}=\frac{\parallel\Omega\parallel^2_{\omega_0}}{\parallel\Omega\parallel^2_{\omega}}
 =e^f \frac{\int_X \omega^n}{\int_X \omega_0^n}.
\end{equation}
Here we denote
\[
    e^f = \|\Omega\|^2_{\omega_0} \int_X \omega_0^n,
\]
and denote
\[
  \det\omega=\det (g_{i\bar j}), \qquad \textup{if \quad $\omega=\frac{\sqrt{-1}}{ 2} \sum_{i,j=1}^n g_{i\bar
j}dz_i\wedge d\bar z_j$}.
\]
At the moment, we write
\begin{equation*}
  \begin{split}
        & \omega_0^{n-1}+\frac{\sqrt{-1}}{ 2}\partial\bar\partial \varphi
         =\Bigl(\frac{\sqrt{-1}}{2}\Bigr)^{n-1}(n-1)! \\
        & \qquad \qquad \cdot \sum_{i,j=1}^n (\Psi_{\varphi})_{i\bar j} s(i,j) dz_1\wedge d\bar z_1 \wedge\cdots\wedge
\widehat{dz_i}\wedge\cdots \widehat{d\bar{z}_j}\wedge \cdots\wedge
dz_n\wedge d\bar z_n.
  \end{split}
\end{equation*}
Here the sign function $s(i,j)$ is equal to $1$ if $i \le j$, and is equal to $-1$ if $i > j$.
By \eqref{eq:omgn} and
\begin{equation*}
  \begin{split}
  \omega^{n-1}
  & = \Bigl(\frac{\sqrt{-1}}{2}\Bigr)^{n-1}(n-1)!\\
  & \quad \; \cdot (\det\omega)
    \sum_{i,j=1}^n g^{i\bar j} s(i,j) dz_1\wedge
d\bar z_1 \wedge\cdots\wedge \widehat{dz_i}\wedge\cdots
\widehat{d\bar z_j}\wedge \cdots\wedge dz_n\wedge d\bar z_n,
  \end{split}
\end{equation*}
we have
\begin{equation*}
(\det\omega) g^{i\bar j}=(\Psi_{\varphi})_{i\bar j}, \qquad \textup{for all $1 \le i, j \le n$}.
\end{equation*}
Hence,
\begin{equation*}
\det\omega=\bigl\{\det \big[ (\Psi_{\varphi})_{i\bar{j}}\big] \bigr\}^{\frac 1
{n-1}}=\bigl\{\det
\big[\omega_0^{n-1}+(\sqrt{-1}/2)\partial\bar\partial\varphi\big]\bigr\}^{\frac
1 {n-1}}.
\end{equation*}
Here $\det [\omega_0^{n-1}+(\sqrt{-1}/2)\partial\bar\partial\varphi]$
stands for the determinant of $n\times n$ matrix of its
coefficients.
Thus, equation (\ref{104}) is equivalent to
\begin{equation} \label{eq:ftCY}
    \frac{\det[\omega_0^{n-1}+(\sqrt{-1}/2)\partial\bar\partial\varphi]}{\det\omega_0^{n-1}}=e^{(n-1)f} \left(\frac{\int_X \omega^n}{\int_X \omega_0^n}\right)^{n-1}.
\end{equation}
 We call the above equation 
 the {\sl form-type Calabi--Yau equation}.
 Clearly, by integrating \eqref{104}, we obtain a compatibility condition
 \begin{equation} \label{eq:cpf}
    \int_X e^f \omega_0^n = \int_X \omega_0^n.
 \end{equation}
Let us denote by $\mathcal{P}(\omega_0)$ the set of all smooth real
$(n-2,n-2)$--forms $\psi$ such that
\begin{equation} \label{eq:defPo}
   \omega_0^{n-1}+\frac{\sqrt{-1}}{2}\partial\bar\partial\psi > 0 \qquad \textup{on $X$}.
\end{equation}
The question is therefore reduced to find, for a given $f \in
C^{\infty}(X)$ with \eqref{eq:cpf}, a smooth real $(n-2,n-2)$--form $\varphi \in
\mathcal{P}(\omega_0)$ satisfying \eqref{eq:ftCY}.


Here is the geometric interpretation of our equation.
Let us briefly recall some definitions related to the hermitian
connection. We follow \cite{Kob}. Let $R$ be the curvature of
hermitian connection with respect to metric $\omega$. Then,
\begin{equation*}
R_{i\bar j k\bar l}=-\frac{\partial^2g_{i\bar j}}{\partial
z_k\partial\bar z_l}+\sum_{p,q=1}^n g^{p\bar q}\frac{\partial g_{i\bar
q}}{\partial z_k}\frac{\partial g_{p\bar j}}{\partial\bar z_l}.
\end{equation*}
We set
\begin{equation*}
R_{k\bar l}= \sum_{i,j=1}^n g^{i\bar j}R_{i\bar jk\bar l},
\end{equation*}
and associate with it a real $(1,1)$-form given by
\begin{equation*}
Ric^h=\sqrt{-1}\sum_{k,l=1}^n R_{k\bar l}dz_k\wedge d\bar z_l.
\end{equation*}
We call $Ric^h$ the \emph{Ricci curvature} of hermitian connection.
Clearly,
\begin{equation*}
Ric^h=\sqrt{-1}\bar\partial\partial\log(\det \omega).
\end{equation*}
So $\parallel\Omega\parallel_\omega = C_0$ is equivalent to the Ricci
curvature $Ric^h=0$.


On the other hand, we can also define the Ricci form $Ric^s$ of the spin connection (i.e. Bismut connection)
on a hermitian manifold. The relation between the two Ricci forms is given by \cite{LYZ}
$$Ric^s=Ric^h-dd^\ast\omega.$$
Here $d^\ast$ is the adjoint operator of $d$ with respect to the metric $\omega$. So when $\omega$ is balanced,
$Ric^s=Ric^h$, and hence, $\parallel\Omega\parallel_\omega = C_0$ is also equivalent
to the Ricci curvature of the spin connection is zero.

In particular, if $\omega_0$ is K\"ahler and let $\varphi$ to be
 \begin{equation*} \label{eq:classic}
   \begin{split}
         \mbox{either} \quad & u\sum_{i=0}^{n-2}\binom{n-1}{i} \bigl(\frac{\sqrt{-1}}{2}\partial\bar\partial u\bigr)^{n-i-2}\wedge\omega_0^i, \quad \mbox{or} \\
         &  -\frac{\sqrt{-1}}{2} \p u \wedge \bar{\p} u \wedge \sum_{i=0}^{n-3} \binom{n-1}{i} \bigl(\frac{\sqrt{-1}}{2}\partial\bar\partial u\bigr)^{n-i-3}\wedge\omega_0^i,
     \end{split}
 \end{equation*}
 then \eqref{eq:ftCY} is reduced to
 \begin{equation*}
            \frac{\det(\omega_0+\frac{\sqrt{-1}}{2}\partial\bar\partial u)}{\det{\omega_0}}=e^f.
 \end{equation*}
 This is the classic equation in the Calabi Conjecture on $c_1(X) =0$,  which was
 settled by Yau~\cite{Yau}.


It seems to us that a form-type equation such as \eqref{eq:ftCY} has
not yet been studied. To begin with, we consider the form-type Calabi--Yau equation
on $T^n$, the complex $n$-torus. Let $(z_1,\ldots, z_n)$ be the complex coordinates on $T^n$ induced from $\mathbb{C}^n$. Then, any non-vanishing holomorphic $n$-form $\Omega$ on $T^n$ is equal to
  \[
      d z_1 \wedge \cdots \wedge d z_n
  \]
  up to multiplying a nonzero constant. We fix such an $n$-form $\Omega$.
  By a \emph{constant form} or a \emph{constant metric} on $T^n$ we mean a differential form or a metric on $T^n$ with constant coefficients. Let $\omega_0$ be a balanced metric on $T^n$. As far as the Bott--Chern cohomology class of $\omega_0^{n-1}$ is concerned, we can assume, without loss of generality, that $\omega_0$ is a constant metric on $T^n$. This is due to the fact that any closed differential form on $T^n$ is cohomologous to a constant form, and the $\partial\bar{\partial}$--Lemma. Our result is as follows:
\begin{theo} \label{th:tori}
  Let $\Omega$ be a non-vanishing holomorphic $n$-form on $T^n$ $(n \ge 3)$, and $\omega_0$ is a constant metric on $T^n$ such that $\|\Omega\|_{\omega_0} = 1$.
   We denote by $C_0$ a positive constant.
  \begin{enumerate}
  \item  \label{it:uni} If $C_0 \le 1$, then for any  metric $\omega$ on $T^n$ such that $[\omega^{n-1}]= [\omega_0^{n-1}] \in H^{n-1,n-1}_{BC}(T^n)$ and that $\|\Omega\|_{\omega} = C_0$, we must have $C_0 = 1$ and
  \[
     \omega = \omega_0.
  \]
  \item \label{it:ext} For each $C_0 > 1$, there exists a metric $\omega$ on $T^n$ such that $[\omega^{n-1}] = [\omega_0^{n-1}]$ and that
  \[
     \|\Omega\|_{\omega} = C_0.
  \]
  \end{enumerate}
\end{theo}
One can see from Theorem~\ref{th:tori} that the normalization constant $C_0$ plays a role here. When $C_0 \le 1$, the theorem tells us that the Calabi--Yau metric is the unique canonical balanced metric. It is the second case, $C_0 > 1$, that marks the difference between a form-type equation and a usual function-type equation. In this case, we establish the existence of a desired balanced metric which is not Calabi--Yau. We further generalize the uniqueness part, Theorem~\ref{th:tori} \eqref{it:uni}, to an arbitrary Calabi--Yau manifold:
\begin{theo} \label{th:CYunique}
  Let $X$ be a compact K\"ahler manifold with a non-vanishing holomorphic $n$-form $\Omega$. Let $\omega_{0}$ be a Calabi--Yau metric such that $\|\Omega\|_{\omega_{0}} = 1$. Then, for any balanced metric $\omega$ on $X$ such that $\omega^{n-1}$ represents the Bott--Chern cohomology class of $\omega_0^{n-1}$ and such that $\|\Omega\|_{\omega} = C_0 \le 1$, we have
  \[
     \omega = \omega_0.
  \]
\end{theo}

For a general case that $\omega_0$ is non-K\"ahler,
one can use the continuity method to solve \eqref{eq:ftCY}.
As an initial step we consider the openness.
Here we have to assume $X$ to be a K\"ahler manifold, endowed
with a K\"ahler metric $\eta$.
For nonnegative integers $k$ and $m$, and a real number $0<
\alpha < 1$, we denote by $C^{k,\alpha}(\Lambda^{m,m}(X))$ the
H\"older space of real $(m,m)$--forms on $X$, and in particular,
$C^{k,\alpha}(\Lambda^{0,0}(X)) \equiv C^{k,\alpha}(X)$.
Let
\[
   \mathcal{F}^{k,\alpha}(X) = \left\{ g \in C^{k,\alpha}(X); \int_X e^g \, \omega_0^n = \int_X \omega_0^n \right \}.
\]
Then $\mathcal{F}^{k,\alpha}(X)$ is a hypersurface in the Banach space $C^{k,\alpha}(X)$.
Let $\omega_0$ be a Hermitian metric on $X$, and $\mathcal{P}(\omega_0)$ be the set given by \eqref{eq:defPo}.
We define a map $M: \mathcal{P}(\omega_0)\cap C^{k+2,\alpha}(\Lambda^{n-2,n-2}(X)) \to \mathcal{F}^{k,\alpha}(X)$ by
\[
   M(\psi) = \log \left(\frac{\omega_{\psi}^n}{\omega_{0}^n}\right) - \log \left(\frac{\int_X \omega_{\psi}^n}{\int_X \omega_0^n} \right),
\]
where, by abuse of notation, $\mathcal{P}(\omega_0)\cap C^{k+2,\alpha}(\Lambda^{n-2,n-2}(X))$
stands for
\[
   \left\{\psi \in C^{k+2,\alpha}(\Lambda^{n-2,n-2}(X)); \omega_0^{n-1} + \ppr \psi >0\right\},
\]
and for each $\psi \in \mathcal{P}(\omega_0)\cap C^{k+2,\alpha}(\Lambda^{n-2,n-2}(X))$, we denote by $\omega_{\psi}$ the positive $(1,1)$--form on $X$ such that
\[
   \omega_{\psi}^{n-1} = \omega_0^{n-1} + \ppr \psi.
\]
Note that equation~\eqref{eq:ftCY} can be written as
\[
   M(\varphi) = f.
\]
\begin{theo} \label{th:openKa}
Let $X$ be an $n$-dimensional K\"ahler manifold $(n \ge 3)$, $\omega_0$ be a Hermitian metric on $X$, $k \ge n+4$ be an integer, and $0<\alpha<1$ be a real number. Given $f \in \mathcal{F}^{k,\alpha}(X)$, suppose that $\varphi \in
\mathcal{P}(\omega_0)\cap C^{k+2,\alpha}(\Lambda^{n-2,n-2}(X))$
satisfies
\begin{equation*}
  M(\varphi) = f.
\end{equation*}
Then, there is a positive number $\delta$, such that for any $g \in
\mathcal{F}^{k,\alpha}(X)$ with $\|g - f\|_{C^{k,\alpha}(X)}\le \delta$, there
exists a function $\psi \in \mathcal{P}(\omega_0)\cap
C^{k+2,\alpha}(\Lambda^{n-2,n-2}(X))$ such that
\[
    M(\psi) = g.
\]
\end{theo}

The rest of the paper is organized as follows: In Section~\ref{se:unique}, we first show Theorem~\ref{th:tori} \eqref{it:uni}.  Next, we prove Theorem~\ref{th:tori} \eqref{it:ext} by explicitly constructing a smooth solution $\varphi\in \mathcal{P}(\omega_0)$ for the form-type equation. These arguments make use of special properties such as the flat structure of $T^n$.
We prove Theorem~\ref{th:CYunique} at the end of Section~\ref{se:unique}. In this respect, we essentially present two proofs for the uniqueness on $T^n$, as they may have interests of their own. In Section~\ref{se:open}, we prove Theorem~\ref{th:openKa} in full details, where one can see the compatibility condition is crucial. Moreover, the approach differs from the standard one in that, the special $(n-2,n-2)$--forms $(u \eta^{n-2})$ are taken, and also in the argument of Proposition~\ref{pr:strong} and Proposition~\ref{pr:weakmax}.

\begin{ack}
The authors would like to thank Professor S.-T.~Yau and also L.-S.
Tseng for helpful discussion.
Part of the work was done while the third named
author was visiting Fudan University, he would like to thank their
warm hospitality. Fu is supported in part by NSFC grants 10771037
and 10831008.
\end{ack}

\section{Uniqueness and Existence} \label{se:unique}
In this section, we adopt the following index convention, unless otherwise indicated. For an $(n-1,n-1)$--form $\Theta$, we denote
\begin{equation*} \label{eq:index}
  \begin{split}
        \Theta & = \Big(\frac{\sqrt{-1}}{2} \Big)^{n-1} (n-1)! \\
            & \quad \cdot \sum_{p,q}s(p,q)\Theta_{p\bar{q}}dz^1\wedge d\bar{z}^1 \cdots
            \wedge\widehat{dz^p} \wedge d \bar{z}^p \wedge\cdots \wedge
            d\bar{z}^q \wedge \widehat{d\bar{z}^q}\wedge\cdots \wedge d z^n \wedge
            d\bar{z}^n,
    \end{split}
\end{equation*}
in which 
\begin{equation} \label{eq:sign}
   s(p,q) =
    \begin{cases}
      - 1,  & \textup{if $p > q$}; \\
      1, & \textup{if $p \le q$}.
    \end{cases}
 \end{equation}
 Here we introduce the sign function $s$ so that, 
 \[
   \begin{split}
    & d z^p \wedge d\bar{z}^q \wedge s(p,q) dz^1\wedge d\bar{z}^1 \cdots
            \wedge\widehat{dz^p} \wedge d \bar{z}^p \wedge\cdots \wedge
            d\bar{z}^q \wedge \widehat{d\bar{z}^q}\wedge\cdots \wedge d z^n \wedge
            d\bar{z}^n \\
        & = dz^1 \wedge d\bar{z}^1 \wedge \cdots \wedge d z^n \wedge d\bar{z}^n, \qquad \textup{for all $1 \le p, q \le n$}.
    \end{split}
 \]
And, if the matrix $(\Theta_{p\bar{q}})$ is invertible, we denote by $(\Theta^{p\bar{q}})$ the transposed inverse of $(\Theta_{p\bar{q}})$, i.e.,
\[
    \sum_l \Theta_{i\bar{l}} \Theta^{j\bar{l}} = \delta_{ij}.
\]
In the following, we may also use the summation convention on repeating indices.
\subsection{Torus case}
Throughout this subsection, we consider $X = T^n$, the complex $n$-torus with $n \ge 3$. We shall prove Theorem~\ref{th:tori}. Note that the first part of Theorem~\ref{th:tori} follows immediately from Lemma~\ref{le:tori} below. We shall prove the second part in Lemma~\ref{le:toriext}.
\begin{lemm} \label{le:tori}
  Let $\omega_0$ be a constant metric on $T^n$. Suppose that there exists an $(n-2,n-2)$--form $\varphi \in \mathcal{P}(\omega_0)$ and a constant $0< C_0 \le 1$ such that
  \begin{equation} \label{eq:unitori}
     C_0 \det \left(\omega_0^{n-1} + \frac{\sqrt{-1}}{2} \partial \bar{\partial}\varphi \right) =  \det \omega_0^{n-1}.
  \end{equation}
  Then, we have $C_0 = 1$ and
  \[
     \pp \varphi = 0.
  \]
\end{lemm}

We need two propositions to derive Lemma~\ref{le:tori}. Let $(z_1,\ldots,z_n)$ be the complex coordinates on $T^n$ induced from $\mathbb{C}^n$.
The corresponding real coordinates are $(x_1,\ldots,x_{2n})$. Here we denote
\begin{equation} \label{eq:zx}
   z_i=x_{2i-1}+\sqrt{-1}x_{2i}, \quad \textup{and hence, $\frac{\partial}{\partial z_i} = \frac{1}{2}\left(\frac{\partial}{\partial x_{2i-1}} - \sqrt{-1}\frac{\partial}{\partial x_{2i}}\right)$},
\end{equation}
for all $1 \le i \le n$. We choose the following volume form on $T^n$:
\begin{eqnarray}
    dV=(\sqrt{-1}/2)^n dz_1 \wedge d\bar{z}_1\wedge \cdots \wedge dz_n\wedge
d\bar{z}_n. \nonumber
\end{eqnarray}

Here are two elementary facts:
\begin{prop} \label{pr:tori1}
For any smooth complex function $f$ defined on $T^n$, we have
\begin{eqnarray}
    \int_{T^n}\frac{\partial^2f}{\partial z_i\partial z_j}dV=0\nonumber, \qquad \textup{for all $i,j=1\cdots,n$.}
\end{eqnarray}
\end{prop}
\begin{proof}
We write
\begin{eqnarray}
    f=f_1+\sqrt{-1}f_2\nonumber,
\end{eqnarray}
where $f_1,f_2$ are real functions on $T^n$. Then,
\begin{eqnarray}
4 \frac{\partial^2f_1}{\partial z_i\partial
z_j}=\frac{\partial^2f_1}{\partial x_{2i-1}\partial
x_{2j-1}}-\frac{\partial^2f_1}{\partial x_{2i}\partial
x_{2j}}-\sqrt{-1}\left(\frac{\partial^2f_1}{\partial x_{2i}\partial
x_{2j-1}}+\frac{\partial^2f_1}{\partial x_{2i-1}\partial
x_{2j}}\right).\nonumber
\end{eqnarray}
We have a similar equation for $f_2$. And note that
\begin{eqnarray}
dV = dx_1\wedge\cdots\wedge dx_{2n}.\nonumber
\end{eqnarray}
The result then obviously follows from the fundamental theorem of calculus.
\end{proof}

\begin{prop} \label{pr:tori2}
Let $B=(b_{i\bar{j}})$ be a hermitian matrix on $T^n$, in which each entry
$b_{i\bar{j}}$ is a complex smooth  function defined on $T^n$ such that
\begin{eqnarray}
\int_{T^n} b_{i\bar{j}} \, dV=0\nonumber.
\end{eqnarray}
Assume that $I+B$ is everywhere positive definite, and there is a constant $c \ge 1$ such that
\begin{eqnarray}
\det(I+B)= c \nonumber \quad \textup{on $T^n$, \; where $I \equiv (\delta_{i\bar{j}})$}.
\end{eqnarray}
Then, $c = 1$ and $B=0$.
\end{prop}
\begin{proof}
Since $I+B$ is  positive definite, we have
\begin{eqnarray} \label{eq:agmtori}
\frac{\text{tr}(I+B)}{n}\geqq \sqrt[n]{\det(I+B)}= \sqrt[n]{c} \qquad \textup{on $T^n$}.
\end{eqnarray}
Integrating \eqref{eq:agmtori} over $T^n$, we obtain
\begin{eqnarray}
\int_{T^n}dV=\int_{T^n}\frac{\text{tr}(I+B)}{n}dV \ge \sqrt[n]{c} \int_{T^n}dV.\nonumber
\end{eqnarray}
Thus, $c=1$, and the inequality of \eqref{eq:agmtori} is in fact an equality. That is,
\begin{eqnarray} \label{eq:eqagmtori}
\frac{\text{tr}(I+B)}{n}= \sqrt[n]{\det(I+B)}=1, \qquad \textup{on $T^n$}.
\end{eqnarray}
Now at an arbitrary point $x$ in $T^n$, we choose a unitary matrix $U$ such that
\begin{eqnarray}
UB\bar{U}^T=\text{dial}\{\lambda_1,\cdots,\lambda_n\}.\nonumber
\end{eqnarray}
Then \eqref{eq:eqagmtori} is equivalent to that
\begin{eqnarray}
1+\lambda_1=1+\lambda_2 = \cdots=1+\lambda_n=1.\nonumber
\end{eqnarray}
This implies that
\begin{eqnarray}
\lambda_i=0, \qquad \textup{for all $i = 1,\ldots,n$}.\nonumber
\end{eqnarray}
Therefore, $B=0$ at $x$. Since $x$ is arbitrary, this finishes the proof.
\end{proof}
Let us now proceed to prove Lemma~\ref{le:tori}:
\begin{pfletori}
Let
\begin{equation*}
  \begin{split}
        \omega_0^{n-1} & = \Big(\frac{\sqrt{-1}}{2} \Big)^{n-1} (n-1)! \\
            & \quad \cdot \sum_{p,q}\Psi_{p\bar{q}} s(p,q) dz^1\wedge d\bar{z}^1 \wedge \cdots
            \wedge\widehat{dz^p}\wedge\cdots \wedge
            \cdots \wedge\widehat{d\bar{z}^q}\wedge\cdots \wedge dz^n \wedge
            d\bar{z}^n.
    \end{split}
\end{equation*}
Here $(\Psi_{i\bar{j}})$ is a constant, positive definite, hermitian matrix, and $s(p,q)$ is given by \eqref{eq:sign}. We can then take a non-degenerate constant matrix $A$ such that
\begin{eqnarray} \label{eq:AOA}
A(\Psi_{i\bar{j}})\bar{A}^T=I.
\end{eqnarray}
We define a hermitian matrix $F_{\varphi} = ((F_{\varphi})_{i\bar{j}})$ on $T^n$ by
\begin{equation*}
  \begin{split}
        \frac{\sqrt{-1}}{2}\partial \bar{\partial}\varphi & = \Big(\frac{\sqrt{-1}}{2} \Big)^{n-1} (n-1)! \\
            & \cdot \sum_{p,q}(F_{\varphi})_{p\bar{q}} s(p,q) dz^1\wedge d\bar{z}^1 \wedge \cdots
            \wedge\widehat{dz^p}\wedge\cdots \wedge
            \cdots \wedge\widehat{d\bar{z}^q}\wedge\cdots \wedge dz^n \wedge
            d\bar{z}^n.
    \end{split}
\end{equation*}
It follows from Proposition~\ref{pr:tori1} that
\begin{eqnarray}
\int_{T^n}(F_{\varphi})_{i\bar{j}}dV=0.\nonumber
\end{eqnarray}
Then, by \eqref{eq:unitori} and \eqref{eq:AOA},
\begin{eqnarray}
\det(I+AF_{\varphi}\bar{A}^T)= C^{-1}_0.\nonumber
\end{eqnarray}
Since $\varphi \in \mathcal{P}(\omega_0)$, we obtain
\begin{eqnarray}
I + AF_{\varphi}\bar{A}^T > 0 \qquad \textup{on $T^n$}.\nonumber
\end{eqnarray}
Applying Proposition~\ref{pr:tori2} yields that $C_0 =1$, and
\begin{eqnarray}
A F_{\varphi} \bar{A}^T=0,\nonumber
\end{eqnarray}
and therefore,
\[
    F_{\varphi} = 0.
\]
\qed
\end{pfletori}

The following lemma establishes the second part of Theorem~\ref{th:tori}. By a linear transformation, if necessary, we can assume the constant metric $\omega_0$ on $T^n$ to be the standard metric:
\begin{eqnarray}
    \omega_0 =\frac{\sqrt{-1}}{2} \big( dz_1 \wedge d\bar{z}_1 +\cdots +dz_n\wedge d\bar{z}_n \big). \nonumber
\end{eqnarray}
\begin{lemm} \label{le:toriext}
For any $0 < \delta < 1$, there exists a smooth $(n-2,n-2)$--form $\varphi \in \mathcal{P}(\omega_0)$ such that
\begin{eqnarray} \label{eq:deltacy}
    \det \left(\omega_0^{n-1}+\frac{\sqrt{-1}}{2}\partial\bar{\partial}\varphi\right)=\delta\det
\omega_0^{n-1}.
\end{eqnarray}
\end{lemm}
\begin{pfletoriext}
We set
\begin{equation} \label{eq:defvarphi}
  \begin{split}
    \varphi
    & =  (n-1)! \left(\frac{\sqrt{-1}}{2}\right)^{n-2} \Big[u(z_1,\bar{z}_1)dz_3\wedge d\bar{z}_3 \wedge \cdots \wedge dz_n \wedge d\bar{z}_n \\
    & \quad  + \, v(z_1,\bar{z}_1)dz_2 \wedge d \bar{z}_2 \wedge \widehat{dz_3} \wedge \widehat{d\bar{z}_3} \wedge d z_4 \wedge d\bar{z}_4 \wedge \cdots \wedge dz_n \wedge d\bar{z}_n \Big].
    \end{split}
\end{equation}
Here $u,v$ are two real, smooth, periodic functions to be determined, with $1 + \Delta u >0$ and $1 + \Delta v > 0$. Since $u$ and $v$ depend only on the first variable, the equation \eqref{eq:deltacy} becomes that
\begin{equation} \label{eq:delta1}
   \Big(1 + \frac{\p^2 u}{\p z_1 \p \bar{z}_1}\Big)\Big(1 + \frac{\p^2 v}{\p z_1 \p \bar{z}_1}\Big) = \delta.
\end{equation}
This reduces to an equation on $T^1$. Note that
\[
    \frac{\p^2 u}{\p z_1 \p \bar{z}_1} = \Delta u, \quad \frac{\p^2 v}{\p z_1 \p \bar{z}_1} = \Delta v,
\]
where $\Delta$ is the standard Laplacian on $T^1$, i.e., the Laplacian associated with $\omega_0|_{T^1}$. We can rewrite \eqref{eq:delta1} as
\begin{equation} \label{eq:delta3}
   1 + \Delta u = \frac{\delta}{1 + \Delta v}.
\end{equation}
Our strategy is to fix a function $v$ and then solve \eqref{eq:delta3} for a function $u$. Note that for a fixed $v$, the necessary and sufficient condition to solve \eqref{eq:delta3} is that
\begin{equation} \label{eq:nesu}
   \int_{T^1} \omega_0|_{T^1} = \delta \int_{T^1} \frac{\omega_0|_{T^1}}{1 + \Delta v}.
\end{equation}

Now let
\begin{equation} \label{eq:defvtori}
    v =  - 4 k \sin\Big(\frac{z_1 + \bar{z}_1}{2} \Big) = - 4k \sin x_1,
\end{equation}
where $0 < k < 1$ is a constant to be determined, and the change of coordinates is given by \eqref{eq:zx}. Then, \eqref{eq:nesu}
becomes that
\[
    \int_{T^1} dx_1 \wedge dx_2 = \int_{T^1} \frac{\delta}{1 + k \sin x_1} dx_1 \wedge dx_2,
\]
that is,
\begin{eqnarray} \label{eq:deltak}
\int^{2\pi}_0\frac{\delta}{1+k\sin x_1}dx_1=2\pi.
\end{eqnarray}
It follows from the proposition below that, for each $0 < \delta < 1$, there exists a real number $0 < k < 1$, depending only on $\delta$, such that \eqref{eq:deltak} holds. Therefore, for $v$ given by \eqref{eq:defvtori}, there is a smooth function $u$, unique up to a constant, satisfies \eqref{eq:delta3}.
Also, by the construction,
\[
   1 + \Delta v > 0, \qquad 1 + \Delta u > 0.
\]
Thus, by \eqref{eq:defvarphi} we obtain an $(n-2,n-2)$--form $\varphi \in \mathcal{P}(\omega_0)$ which solves \eqref{eq:deltacy}.
\qed
\end{pfletoriext}
\begin{prop}
    Let \begin{eqnarray}
Z(k)= \frac{1}{2\pi}\int^{2\pi}_0\frac{1}{1+k\sin x}dx, \qquad \textup{for all $0 \le k < 1$}.
\end{eqnarray}
Then, for any $0 < \delta < 1$, there exists a unique number $0< k_{\delta} < 1$ such that
\[
   Z(k_{\delta}) = \delta^{-1}.
\]
\end{prop}
\begin{proof}
  Clearly, the function $Z$ is smooth on $0 \le k < 1$. Note that $Z(0) = 1$, and that
  \[
    \begin{split}
     Z(k)
     & \ge \frac{1}{2\pi}\int_{3\pi/2}^{2\pi} \frac{dx}{1 + k \sin x} \\
     & = \frac{1}{\pi \sqrt{1-k^2}} \arctan \sqrt{\frac{1+k}{1-k}}
     \to +\infty, \quad \textup{as $k \to 1^-$}.
     \end{split}
  \]
  The existence then follows from the intermediate value theorem in calculus. The uniqueness is due to the monotonicity of $Z$ on $[0,1)$, which is readily seen by verifying $Z'(0) = 0$ and $Z''(k)> 0$ on $[0,1)$.
\end{proof}
  Lemma~\ref{le:toriext} can be easily generalized to the case of the product of a compact hermitian manifold with $T^k$, $k\ge 3$. See the corollary below:
\begin{coro}
  Let $M^n = N^{n-k} \times T^k$, $k \ge 3$, where $(N^{n-k},\omega_N)$ is an $(n-k)$-dimensional compact hermitian manifold.
    We denote $\omega = \omega_N + \omega_0$, where $\omega_0$ is a constant metric on $T^k$. Then, for any $1 > \delta > 0$, there exists a smooth $(n-2,n-2)$--form $\psi \in \mathcal{P}(\omega)$ on $M$ such that
      \[
         \det \left( \omega^{n-1} + \ppfr \psi \right) = \delta \det (\omega^{n-1}).
      \]
\end{coro}
\begin{proof}
  Let
  \[
     \psi = \omega_N^{n-k} \wedge \varphi,
  \]
  where $\varphi$ is the $(k-2,k-2)$--form on $T^k$ obtained by Lemma~\ref{le:toriext}, i.e., $\varphi \in \mathcal{P}(\omega_0)$ satisfies that
  such that
      \[
         \det \left( \omega_0^{k-1} + \ppfr \varphi \right) = \delta \det (\omega_0^{k-1}).
      \]
      Then, obviously $\psi$ satisfies the requirement.
\end{proof}

\subsection{K\"ahler case}
In this subsection, we shall prove Theorem~\ref{th:CYunique}. Observe that it is sufficient to prove the following lemma.
\begin{lemm}
Let $(X,\omega_0)$ be a compact K\"ahler manifold. Consider
\[
   \det \left(\omega_0^{n-1} + \ppfr \varphi\right) = C_1 \det \omega_0^{n-1},
\]
where $\varphi \in \mathcal{P}(\omega_0)$, and $C_1>0$ is a constant. If $C_1 \ge 1$, then
\[
   \pp \varphi = 0.
\]
\end{lemm}
\begin{proof}
By a direct calculation, since $\omega_0$ is K\"ahler, we have
 \begin{eqnarray}\label{eq:001}
\int_X (\omega_0^{n-1})^{i\bar{j}} \Big(\ppfr \varphi\Big)_{i\bar{j}} \, \omega_0^n \notag
      = n\int_X \omega_0 \wedge \Big(\ppfr \varphi\Big) =0.
  \end{eqnarray}
Similar to the torus case, we apply the arithmetic--geometric mean inequality to obtain
  \begin{equation} \label{eq:agm}
  \begin{split}
     C_1^{1/n} & = \left[\frac{\det (\omega_{\varphi}^{n-1})}{\det (\omega_0^{n-1})}\right]^{1/n} \\
       & \le 1 + \frac{1}{n} \sum_{i,j} (\omega_0^{n-1})^{i\bar{j}} \Big(\ppfr \varphi\Big)_{i\bar{j}}.
  \end{split}
  \end{equation}
  Integrating over $X$ with respect to $\omega_0$ and using first equality yields that
  \[
     C_1^{1/n}\int_X \omega_0^n \le  \int_X \omega_0^n.
  \]
  This shows that $C_1=1$ and we must have a pointwise equality in \eqref{eq:agm}. This forces that
  \[
     \ppfr \varphi = 0.
  \]
\end{proof}

\section{Openness} \label{se:open}

Let $(X,\eta)$ be a K\"ahler manifold, and $\omega_0$ be a Hermitian metric on $X$. Given $f \in C^{\infty}(X)$, we would like to study the solution $\varphi \in \mathcal{P}(\omega_0)$ of the following equation
\begin{equation} \label{eq:ftCYKa}
    \frac{\omega_{\varphi}^n}{\omega_0^n} =  \frac{e^f}{V} \int_X \omega_{\varphi}^n.
\end{equation}
Here $\omega_{\varphi}$ is a positive $(1,1)$--form on $X$ such that
\[
  \omega_{\varphi}^{n-1} = \omega_0^{n-1} + \ppr \varphi,
\]
and
\[
   V = \int_X \omega_0^n.
\]
Equation~\eqref{eq:ftCYKa} is the same as \eqref{104}, which is equivalent to the form-type Calabi--Yau equation \eqref{eq:ftCY}. A compatibility condition for \eqref{eq:ftCYKa} is
\[
   \int_X e^f \omega_0^n = V.
\]

In what follows, we fix $k$ to be an integer greater than $n+3$, and fix a real number $\alpha$ with $0 < \alpha <1$. We denote by $C^{k,\alpha}(X)$ the usual H\"older space of real-valued functions on $X$. Recall that
\[
   \mathcal{F}^{k,\alpha}(X) = \left\{ g \in C^{k,\alpha}(X); \int_X e^g \, \omega_0^n = V \right \},
\]
which is a hypersurface in the Banach space $C^{k,\alpha}(X)$. For any $\psi$ contained in the intersection of $\mathcal{P}(\omega_0)$ and $C^{k+2,\alpha}(\Lambda^{n-2,n-2}(X))$,
\[
   M(\psi) \equiv \log \frac{\omega_{\psi}^n}{\omega_{0}^n} - \log \left(\frac{1}{V}\int_X \omega_{\psi}^n \right) \in \mathcal{F}^{k,\alpha}(X).
\]
By the map $M$, equation~\eqref{eq:ftCYKa} can be rewritten as
\[
   M(\varphi) = f.
\]

To prove Theorem~\ref{th:openKa}, we first compute the linearization of $M$.
\begin{prop} \label{pr:le}
   Let $G(\varphi) = \omega_{\varphi}^n$ for all $\varphi \in \mathcal{P}(\omega_0)$,
and denote by $G_{\varphi}$ the Fr\'echet derivative of $G$ at $\varphi$. Then, given $\varphi \in \mathcal{P}(\omega_0)$, we have
   \begin{equation*}
            G_{\varphi}(\psi)=\frac {n\sqrt{-1}}{2(n-1)}
            \partial\bar\partial\psi\wedge\omega_\varphi,
    \end{equation*}
    for all $\psi \in C^{k+2,\alpha}(\Lambda^{n-2,n-2}(X))$.
\end{prop}
\begin{proof}
    For any real $(n-2,n-2)$--form $\psi$,
    \begin{align}
        G_{\varphi}(\psi)
        & =\left.\frac{d}{ds}\left(\omega_{\varphi+s\psi}^n\right)\right|_{s=0} \notag \\
        & =n\omega_\varphi^{n-1}\wedge \left.\frac{d}{ds}(\omega_{\varphi+s\psi})\right|_{s=0} \label{eq:lin1}\\
        &=\left.\frac{d}{ds}(\omega^{n-1}_{\varphi+s\psi})\right|_{s=0}\wedge\omega_\varphi+\omega_\varphi^{n-1}\wedge
\left.\frac{d}{ds}(\omega_{\varphi+s\psi})\right|_{s=0} \notag \\
        & =(\sqrt{-1}/2)
\partial\bar\partial\psi\wedge\omega_{\varphi}+\omega_\varphi^{n-1}\wedge\left.\frac{d}{ds}(\omega_{\varphi+s\psi})\right|_{s=0}. \label{eq:lin2}
\end{align}
    Comparing \eqref{eq:lin1}  with \eqref{eq:lin2}, we obtain that
    \begin{equation*}
        G_{\varphi}(\psi)=\frac {n}{n-1}(\sqrt{-1}/2)\partial\bar\partial\psi\wedge\omega_\varphi.
    \end{equation*}
\end{proof}
\begin{coro} \label{co:le}
  For any $\varphi \in \mathcal{P}(\omega_0)$, the Fr\'echet derivative of $M$ at $\varphi$ is given by
  \[
   M_{\varphi}(\psi) = \frac{n(\sqrt{-1}/2)
\partial\bar\partial\psi\wedge\omega_\varphi}{(n-1)\omega_{\varphi}^n} - \frac{n \int_X (\sqrt{-1}/2)\partial\bar\partial \psi \wedge \omega_{\varphi}}{(n-1)\int_X \omega_{\varphi}^n},
\]
for all $\psi \in C^{k+2,\alpha}(\Lambda^{n-2,n-2}(X))$.
\end{coro}

Next, we recall the Local Surjectivity Theorem (see \cite[p. 175 and p. 108]{AMR}, for example).
\begin{LST}
Let $\mathcal{E}$ and $\mathcal{F}$ be Banach manifolds, and $U \subset
\mathcal{E}$ be an open subset. If $\mathfrak{F} : U \to
\mathcal{F}$ is a $C^1$ map, and $\mathfrak{F}_{\xi} \equiv
D\mathfrak{F}(\xi)$ is onto from the tangent space $T_{\xi} \mathcal{E}$
to the tangent space $T_{\mathfrak{F}(\xi)} \mathcal{F}$, then
$\mathfrak{F}$ is locally onto; that is, there exist open
neighborhoods $U_1$ of $\xi$ and $V_1$ of $\mathfrak{F}(\xi)$ such
that $\mathfrak{F}|_{U_1} : U_1 \to  V_1$ is onto.
\end{LST}
Thus, to show Theorem~\ref{th:openKa}, it suffices to show that the linearization $M_{\varphi}$ is surjective from $C^{k+2,\alpha}(\Lambda^{n-2,n-2}(X))$ to $T_f \mathcal{F}^{k,\alpha}(X)$, which denotes the tangent space of $\mathcal{F}^{k,\alpha}(X)$ at $f$. Now let us introduce the space
\[
   \mathcal{E}^{k,\alpha}(X) = \left\{ h \in C^{k,\alpha}(X); \int_X h \, \omega_{\varphi}^n = 0 \right\}.
\]
Note that $\mathcal{E}^{k,\alpha}(X)$ is itself a Banach space, as a closed subspace in $C^{k,\alpha}(X)$. There is another point of view: We can define an equivalence relation on the elements in $C^{k,\alpha}(X)$ by
\[
   h \sim g \quad \mbox{if and only if $h - g \equiv$ some constant.}
\]
In this regard, $\mathcal{E}^{k,\alpha}(X) = C^{k,\alpha}(X)/\sim$. Observe that
\[
   T_f \mathcal{F}^{k,\alpha}(X) =  \mathcal{E}^{k,\alpha}(X).
\]

To prove the surjectivity of $M_{\varphi}$, we consider a special class of the $(n-2,n-2)$--forms, that is,
\begin{equation} \label{eq:ueta}
   \psi = u \eta^{n-2},  \qquad \textup{where $u \in \mathcal{E}^{k+2,\alpha}(X)$}.
\end{equation}
We recall that $\eta$ is the K\"ahler metric on $X$. For simplicity we denote
\[
   L(u) = M_{\varphi}(u\eta^{n-2}).
\]
Then, by Corollary~\ref{co:le},
\begin{equation} \label{eq:Linear}
    L u  = \frac{n(\sqrt{-1}/2)\p \bar{\p} u \wedge \eta^{n-2} \wedge \omega_{\varphi}}{(n-1)\omega_{\varphi}^n}
    - \frac{n \int_X \ppr u \wedge \eta^{n-2} \wedge \omega_{\varphi}}{(n-1)\int_X \omega_{\varphi}^n}.
\end{equation}
We shall prove the following result:
\begin{lemm} \label{le:ism}
Let $k \ge n + 4$, and $0 < \alpha < 1$. For any $h \in \mathcal{E}^{k,\alpha}(X)$, there exists a unique function $u \in \mathcal{E}^{k+2,\alpha}(X)$  satisfying that
\begin{equation} \label{eq:KaL}
   L u = h.
\end{equation}
\end{lemm}
Lemma~\ref{le:ism} implies that $M_{\varphi}: C^{k+2,\alpha}(\Lambda^{n-2,n-2}(X)) \to \mathcal{E}^{k,\alpha}(X)$ is surjective, and hence, Theorem~\ref{th:openKa} follows.

The rest of this section is devoted to prove Lemma~\ref{le:ism}. We denote by $W^{k,p}(\Omega, \omega_{\varphi})$ the usual Sobolev space with respect to $\omega_{\varphi}$ on a domain $\Omega$ in $X$.
In the rest of this section, we may denote $W^{k,p}(\Omega) = W^{k,p}(\Omega, \omega_{\varphi})$ for simplicity; furthermore, when $\Omega = X$, we abbreviate $W^{k,p} = W^{k,p}(X) = W^{k,p}(X,\omega_{\varphi})$. Notice that $W^{0,2}(X) \equiv L^2(X)$.



We introduce the following spaces:
\[
   \mathcal{H}  = \left\{ v \in W^{1,2}(X); \int_X v \; \omega_{\varphi}^n = 0 \right\},
\]
and
\[
     \mathcal{L} = \left\{ v \in L^2(X); \int_X v \; \omega_{\varphi}^n = 0 \right\}.
\]
Clearly, $\mathcal{H}$ and $\mathcal{L}$ are Hilbert spaces, as closed subspaces in $W^{1,2}(X)$ and $L^2(X)$, respectively.
We define a bilinear map $A : \mathcal{H} \times \mathcal{H} \to \mathbb{R}$ by
\[
  \begin{split}
   A (u, v)
   & = \frac{n\sqrt{-1}}{4(n-1)} \int_X  \eta^{n-2} \wedge \omega_{\varphi} \wedge \big(\p u \wedge \bar{\p} v + \p v \wedge \bar{\p} u\big) \\
   &  \quad + \frac{n\sqrt{-1}}{4(n-1)} \int_X v \eta^{n-2} \wedge \big(\p u \wedge \bar{\p}\omega_{\varphi} + \p \omega_{\varphi} \wedge \bar{\p} u \big).
  \end{split}
\]
\begin{defi}
Given $h \in \mathcal{L}$, we say that $u \in \mathcal{H}$ is a \emph{weak} solution of the equation
\begin{equation} \label{eq:weakKaLh}
   - L u = h,
\end{equation}
if $u$ satisfies that
\begin{equation} \label{eq:weakA}
   A(u, v) = \int_X h v \, \omega_{\varphi}^n \equiv  \langle h, v \rangle_{L^2}, \qquad \textup{for all $v \in \mathcal{H}$}.
\end{equation}
\end{defi}
Let us remark that, if $u$ is a \emph{classical} solution of \eqref{eq:weakKaLh}, i.e., $u \in C^2(X)$, then one can obtain \eqref{eq:weakA} by integrating \eqref{eq:weakKaLh} by parts with respect to $\omega_{\varphi}^n$. Conversely, we have the following result:
\begin{prop} \label{pr:strong}
 If $u \in C^3(X)$ satisfies \eqref{eq:weakA} for some $h \in C^1(X)\cap \mathcal{L}$, then
\[
   - Lu = h.
\]
\end{prop}
\begin{proof}
 First, 
 we claim the following fact: If $\chi \in C^{1}(X)$ satisfy that
   \begin{equation} \label{eq:const}
     \int_X \chi v \, \omega_{\varphi}^n = 0, \qquad \textup{for all $v \in \mathcal{H}$},
   \end{equation}
   then $\chi$ is a constant function on $X$. To see this, let
   \[
      v = \chi - \frac{\int_X \chi \omega_{\varphi}^n}{\int_X \omega_{\varphi}^n};
   \]
   then $v \in \mathcal{H}$ and \eqref{eq:const} implies that
   \[
      \int_X |v|^2 \omega_{\varphi}^n = 0.
   \]
   This proves the claim. It follows that
\[
    \frac{n(\sqrt{-1}/2)\p \bar{\p} u \wedge \eta^{n-2} \wedge \omega_{\varphi}}{(n-1)\omega_{\varphi}^n}
    - h = \mbox{some constant}.
\]
Thus, integrating with respect to $\omega_{\varphi}^n$ yields the result.
\end{proof}
The following weak maximum principle is similar to that on a domain in the Euclidean space (see, for example, Gilbarg--Trudinger~\cite[p. 179]{GT}). Proposition~\ref{pr:weakmax} is trivial, if $d \omega_{\varphi} = 0$.
\begin{prop} \label{pr:weakmax}
   Suppose that $u \in \mathcal{H}$ satisfies
   \begin{equation} \label{eq:weakmax}
      A(u, v) = 0, \qquad \textup{for all $v \in \mathcal{H}$}.
   \end{equation}
   Then, $u = 0$.
\end{prop}
\begin{proof}
  It suffices to prove $\sup_X u \le 0$, as one can then replace $u$ by $-u$. (Here $\sup$ stands for the essential supremum.) Suppose the contrary. Take a constant $\delta$ such that $0 < \delta < \sup_X u$, and define
  \begin{equation} \label{eq:defv}
     v = (u - \delta)^+ - \frac{\int_X (u - \delta)^+ \omega_{\varphi}^n}{\int_X \omega_{\varphi}^n},
  \end{equation}
  in which $(u - \delta)^+ = \max\{ u- \delta, 0\}$. Then, $v \in \mathcal{H}$ and
  \[
    d v = d(u - \delta)^+ =
  \begin{cases}
    d u, & \textup{if $u > \delta$}, \\
    0,   & \textup{if $u \le \delta$}.
  \end{cases}
  \]
  Let us denote by $\Gamma$ the compact support of $dv$. Then, we obtain by \eqref{eq:weakmax} and metric equivalence of $\eta$, $\omega_{\varphi}$, that
  \[
     \| \nabla v\|^2_{L^2} = \int_{\Gamma} |\nabla v|^2 \omega_{\varphi}^n \le C \int_{\Gamma} |v| |\nabla v| \omega_{\varphi}^n.
  \]
  Here and below, we denote by $C$ a generic positive constant depending only on $\eta$, $\omega_{\varphi}$, and $n$. Apply H\"older's inequality to get
  \begin{equation} \label{eq:invholder}
      \| \nabla v \|_{L^2} \le C \| v \|_{L^2(\Gamma)}.
  \end{equation}
  On the other hand, combining the Sobolev inequality and Poincar\'e inequality yields that
  \begin{equation}\label{eq:SP}
      \| v \|_{L^{2n/(n-1)}} \le C (\| \nabla v \|_{L^2} + \| v\|_{L^2}) \le C \| \nabla v \|_{L^2}.
  \end{equation}
  Hence, by \eqref{eq:invholder} and \eqref{eq:SP},
  \[
      \| v\|_{L^{2n/(n-1)}} \le C \| v \|_{L^2(\Gamma)} \le C |\Gamma|^{\frac{1}{2n}} \| v \|_{L^{2n/(n-1)}},
  \]
  in which $|\Gamma|$ denotes the measure of $\Gamma$ with respect to $\omega_{\varphi}$. It follows that
  \begin{equation} \label{eq:low}
     |\Gamma| = |\{ u > \delta, |du| > 0\}| \ge C^{-1}.
  \end{equation}
  Letting $\delta$ tend to $\sup u$ implies that $|du|>0$ on a set of positive measure in $\{x \in X; u(x) = \sup_X u\}$, which is evidently impossible by Lemma 7.7 in Gilbarg--Trudinger~\cite[p. 152]{GT}. This proves that $\sup u \le 0$.
\end{proof}

The next two propositions are standard, for which we need the Lax--Milgram Theorem (see Evans~\cite[p. 297]{Evans}, for example) and the Fredholm alternative (see \cite[p. 641]{Evans} for example). We include them here for completeness.
\begin{LMthm}
  Let $H$ be a real Hilbert space, and $I : H \times H \to \mathbb{R}$ be a bilinear mapping. Assume that, there exist positive constants $\beta$ and $\mu$ such that
  \[
     |I(u,v)| \le \beta \|u\| \|v\|, \qquad \textup{for all $u, v \in H$},
  \]
  and
  \[
     I(v,v) \ge \mu \|v\|^2, \qquad \textup{for all $v \in H$}.
  \]
  Then, for any bounded linear functional $f$ on $H$, there exists a unique element $u \in H$ satisfying that
  \[
     I(u, v) = f(v) \qquad \textup{for all $v \in H$}.
  \]
\end{LMthm}
\begin{FAthm}
  Let $E$ be a Banach space and $K : E \to E$ be a compact linear operator. Then,
  \[
     \textup{$\ker (I - K) = \{0\}$ \quad if and only if \quad $\im (I - K) = E$},
  \]
  where $I: E \to E$ is the identity operator.
\end{FAthm}
\begin{prop} \label{pr:Ka1st}
  There exists a nonnegative constant $\gamma$, depending on $\omega_{\varphi}$ and $\eta$, such that for any $h \in \mathcal{L}$, there exists a unique weak solution $u \in \mathcal{H}$ of
  \begin{equation} \label{eq:Ka1stweak}
     - L_{\gamma}u \equiv - L u + \gamma u = h.
  \end{equation}
  That is, the function $u$ satisfies
  \begin{equation} \label{eq:Ka1st}
     A (u, v) + \gamma \langle u, v\rangle_{L^2} = \langle h, v \rangle_{L^2}, \qquad \textup{for all $v \in \mathcal{H}$}.
  \end{equation}
\end{prop}
\begin{proof}
 We have, by the metric equivalence of $\eta$ and $\omega_{\varphi}$,
 \[
    |A(u,v)| \le \beta \|u\|_{W^{1,2}} \|v \|_{W^{1,2}},
 \]
 and
 \[
    A(u,u) + \gamma \|u\|_{L^2} \ge \mu \|u\|_{W^{1,2}}.
 \]
 Here $\beta> 0$, $\gamma \ge 0$, and $\mu > 0$ are constants depending only on $\eta$ and $\omega_{\varphi}$. The result then follows from applying Lax--Milgram Theorem to
 \[
    I (u, v) = A(u, v) + \gamma \langle u , v\rangle_{L^2}, \qquad \textup{for all $u, v\in \mathcal{H}$}.
 \]
\end{proof}

\begin{prop} \label{pr:Kamain}
  For any $h \in \mathcal{L}$, there exists a unique weak solution $u \in \mathcal{H}$ of
  \[
     - Lu = h.
  \]
\end{prop}
\begin{proof}
  By Proposition~\ref{pr:Ka1st} we can define a map $L_{\gamma}^{-1}: \mathcal{L} \to \mathcal{H}$ as follows: For each $f \in \mathcal{L}$, we define $L_{\gamma}^{-1}(f)$ to be the unique function $w \in \mathcal{H}$ satisfying
  \[
     A(w, v) + \gamma \langle w, v \rangle_{L^2} = \langle f, v \rangle_{L^2}.
  \]
  Clearly, $L_{\gamma}^{-1}$ is linear, and is a compact operator from $\mathcal{L}$ to $\mathcal{L}$, in view of Rellich Theorem.
  To prove the result, it suffices to show that, for a given $h \in \mathcal{L}$, there exists a unique $u \in \mathcal{L}$ satisfying that
  \[
     u = L_{\gamma}^{-1}( h + \gamma u).
  \]
  Equivalently, we need to solve a unique $u \in \mathcal{L}$ for the following equation:
  \[
     (I - \gamma L_{\gamma}^{-1})u = L^{-1}_{\gamma} h.
  \]
  To invoke the Fredholm alternative, we turn to the kernel of $(I - \gamma L_{\gamma}^{-1})$ in $\mathcal{L}$, i.e.,
  \[
     \{u \in \mathcal{L} ; \, u - \gamma L_{\gamma}^{-1} u = 0\}. 
  \]
  This is equivalent to investigate the function $u \in \mathcal{H}$ such that
  \[
     A(u, v) = 0 \qquad \textup{for all $v \in \mathcal{H}$}.
  \]
  By Proposition~\ref{pr:weakmax}, $u = 0$. The result then follows from the Fredholm alternative.
\end{proof}

Now we are in a position to prove Lemma~\ref{le:ism}:
\begin{pflemism}
The uniqueness of \eqref{eq:KaL} is an immediate consequence of Proposition~\ref{pr:Kamain}, since a $C^2$ solution of \eqref{eq:KaL} is in particular a weak solution of $- L u = - h$.

Given $h \in C^{k,\alpha}(X)$, we have $h \in W^{k,2}(X)$, since $X$ is compact. Then, by Proposition~\ref{pr:Kamain}, equation~\eqref{eq:KaL} has a weak solution $u \in W^{1,2}(X)$. Then, we obtain
\[
        u \in W^{k+2,2}(X),
\]
by the local regularity theorem (see, for example, Evans~\cite[p. 314]{Evans} or Gilbarg--Trudinger~\cite[p. 186]{GT}).
Since $k \ge n + 4$, $k - 2n/2 -1 \ge 3$. We apply the Sobolev imbedding theorem to obtain that
\[
    u \in C^{3}(X).
\]
By Proposition~\ref{pr:strong}, $u$ is the classical solution for \eqref{eq:KaL}.
It follows from the bootstrap argument (\cite[p. 109]{GT}) that
\[
   u \in C^{k+2,\alpha}(X).
\]
\qed
\end{pflemism}

\end{document}